\documentclass{amsart}
\usepackage{amsfonts,amsmath,amsthm,amscd,amssymb,latexsym}
\usepackage[english]{babel}
\usepackage[all]{xy}

\DeclareMathOperator{\pa}{PAut}
\DeclareMathOperator{\id}{id}
\DeclareMathOperator{\dom}{dom}
\DeclareMathOperator{\ran}{ran}
\DeclareMathOperator{\rank}{rank}
\DeclareMathOperator{\rk}{rk}

\newcommand{\pwr}{\,{\wr_p}\,}
\newcommand{\set}[1]{\left\{#1\right\}}
\newcommand{\ra}{\rightarrow}
\newcommand{\isd}{\mathcal{I}_2}
\newcommand{\is}[1]{\mathcal{I}_{#1}}
\newcommand{\abs}[1]{\left\vert#1\right\vert}
\newcommand{\pn}[1]{\mathcal{P}_{#1}}

\begin{document}

\title[Spectral properties of partial automorphisms]{Spectral properties of partial automorphisms of  binary rooted tree}
\author{Eugenia Kochubinska}

\address{Taras Shevchenko National University of Kyiv, Faculty of Mechanics and Mathematics, Volodymyrska str. 64, 01601, Kyiv, Ukraine.}

\theoremstyle{plain}
\newtheorem{theorem}{Theorem}
\newtheorem{lemma}{Lemma}
\newtheorem{proposition}{Proposition}
\newtheorem{corollary}{Corollary}
\newtheorem{definition}{Definition}
\theoremstyle{definition}
\newtheorem{example}{Example}
\newtheorem{remark}{Remark}
\begin{abstract}
	We study asymptotics of the spectral measure of a randomly chosen partial automorphism of a rooted tree.  To every partial automorphism $x$ we assign its action matrix $A_x$.
	It is shown that the uniform distribution  on eigenvalues of $A_x$ converges weakly in probability to $\delta_0$ as $n \to \infty$, where $\delta_0$ is the delta measure concentrated at $0$.
\end{abstract}

\maketitle

\section*{Introduction}
We consider semigroup of partial automorphisms of a binary $n$-level rooted tree.  Throughout the paper by a partial automorphism we mean root-preserving  injective tree homomorphism defined on a connected subtree. This semigroup was studied, in particular, in \cite{comb, edm}

We are interested in spectral properties of this  semigroup. The similar question for an  automorphism group was studied in \cite{Evans}.  Evans assigned equal probabilities to the eigenvalues of a randomly chosen  automorphism  of a regular rooted tree, and considered the random measure $\Theta_n$ on the unit circle $C$.
 He has shown that $\Theta_n$   converges weakly in probability to $\lambda$ as $n\to\infty$, where $\lambda$ is the normalized Lebesgue measure on the unit circle. 
 
Let $B_n=\set{v_i^n\mid i=1, \ldots, 2^n}$ be the set of vertices of the $n$th level of the $n$-level binary rooted tree.  To a randomly chosen partial automorphism $x$, we assign the action matrix   $A_x = \left(\mathbf{1}_{\set{x (v_i^n) = v_j^n }}\right)_{i,j=1}^{2^n}.$ 
 Let $$\Xi_n = \frac{1}{2^n} \sum\limits_{k=1}^{2^n} \delta_{\lambda_k}$$ be the uniform distribution on eigenvalues of $A_x$. We show that $\Xi_n$ converges weakly in probability to $\delta_0$ as $n \to \infty$, where $\delta_0$ is the delta measure concentrated at $0$.
 
 The remaining of the paper is organized as follows. Section 2 contains basic facts on partial wreath product of semigroup and its connection with a semigroup of partial automorphisms of a regular rooted tree. The main result is stated and proved in Section 3.

\section{Preliminaries}\label{sec:basic}
For a set $X=\set{1, 2}$ consider the set $\is{2}$ of all partial bijections. List all of them using standard tableax  representation: $$\set{\left(\begin{matrix} 1 &2\\1&2 \end{matrix}\right), \left(\begin{matrix} 1 &2\\2&1 \end{matrix}\right), \left(\begin{matrix} 1 &2\\1&\varnothing \end{matrix}\right), \left(\begin{matrix} 1 &2\\\varnothing&2 \end{matrix}\right), \left(\begin{matrix} 1 &2\\2&\varnothing \end{matrix}\right), \left(\begin{matrix} 1 &2\\\varnothing&1 \end{matrix}\right), \left(\begin{matrix} 1 &2\\\varnothing&\varnothing \end{matrix}\right)}.$$   This set forms an inverse semigroup under natural composition law, namely, $f \circ g:
\dom(f) \cap f^{-1} \dom(g) \ni x \mapsto g(f(x))$ for $f, \; g \in
\mathcal{I}_2$. Obviously,  $\is{2}$ is a particular case of the well-known inverse symmetric semigroup. Detailed description of it can be found in \cite[Chapter 2]{GM}.

Recall the definition of a partial wreath product of semigroups.  Let $S$ be an arbitrary semigroup. For functions $f\colon \dom(f) \ra S$, $g\colon \dom(g)\ra S$ define 
 the product $fg$  as:
\begin{equation*}
\dom(fg)=\dom(f)\cap\dom(g),\  (fg)(x)=f(x)g(x) \text{\ for all } x\in
\dom(fg).
\end{equation*}
For $a\in \isd, f\colon \dom(f)\ra S$, define $f^a$ as:
\begin{equation*}
\begin{gathered}
(f^a)(x)=f(x^a),\ \dom(f^a)=\{x \in \dom(a); x^a\in \dom(f)\}.
\end{gathered}
\end{equation*}

\begin{definition}
The partial wreath square of semigroup $\isd$ is  the
set $$\set{(f,a)\mid a \in \isd, f\colon  \dom(a)\ra \isd }$$ with
composition defined by
$$(f,a)\cdot (g,b)=(fg^a,ab)$$
Denote it by $\isd \pwr \isd$.
\end{definition}

The partial wreath square of $\is{2}$  is a semigroup, moreover, it
is an inverse semigroup \cite[Lemmas 2.22 and 4.6]{Meldrum}. We
may recursively define any  partial wreath power of the finite
inverse symmetric semigroup.  Denote by $\pn{n}$ the $n$th partial wreath power of $\is{2}$.

\begin{definition}
The partial wreath $n$-th power of semigroup $\isd$ is defined as
a semigroup $$\pn{n}=\big(\pn{n-1} \big) \pwr \isd =\set{(f,a)\mid a \in \isd, \ f\colon \dom(a)\ra \pn{n-1}}$$ with composition
defined by
$$(f,a) \cdot (g,b)=(fg^a,ab),$$ where  $\pn{n-1}$ is the
partial wreath $(n-1)$-th power of semigroup $\isd$
\end{definition}

\begin{proposition}\label{proposition:card}
Let $N_n$ be the number of elements in the semigroup $\pn{n}$. Then $N_n= 2^{2^{n+1}-1}-1$
\end{proposition}
\begin{proof}
We proceed by induction.

If $n=1$, then $2^{2^2-1}-1=7$. This is exactly the number of elements in $\is{2}$.

Assume that $N_{n-1}=2^{2^{n}-1}-1$. Then \begin{gather*}
N_n=\abs{\set{(f,a)\mid a\in \is{2}, f\colon \dom(a)\ra N_{n-1}}}\\=\sum_{a\in \is{2}}N_{n-1}^{\abs{\dom(a)}}=\sum_{a\in \is{2}}\left(2^{2^n-1}-1\right)^{\abs{\dom(a)}}\\=1+4\cdot(2^{2^n-1}-1)+2\cdot(2^{2^n-1}-1)^2\\=1+4\cdot 2^{2^n-1}-4+2\cdot 2^{2^{n+1}-2}-4\cdot 2^{2^n-1}+2=2^{2^{n+1}}-1.\qedhere
\end{gather*}
\end{proof}

\begin{remark}
Let $T$ be an $n$-level binary rooted tree. We define a partial automorphism of a tree $T$ as an isomorphism $x: \Gamma_1\to \Gamma_2$ of  subtrees $\Gamma_1$ and $\Gamma_2$ of $T$ containing root. Denote $\dom(x):=\Gamma_1$, $\ran (x):= \Gamma_2$ domain and image of $x$ respectively. Let $\pa T$ be the set of all partial automorphisms of $T$. Obviously, $\pa T$ forms a semigroup under natural composition law.   It was proved in \cite[Theorem 1]{comb} that the  partial wreath power $\pn{n}$ is isomorphic to $\pa T$. 
\end{remark}
\section{Asymptotic behaviour of a spectral measure of a binary rooted tree}

 We identify $x\in \pn{n}$ with a partial automorphism from $\pa T$.
  Recall, that  $B_n$ denotes the set of vertices of the $n$th level of $T$. Clearly, ${\abs{B_n}=2^n}$. 
Let us enumerate the vertices of $B_n$ by positive integers from  1 to $2^n$: $$B_n=\set{v_i^n\mid i=1, \ldots, 2^n}.$$ To a randomly chosen transformation $x \in \mathcal{P}_n$, we assign the matrix   $$A_x = \left(\mathbf{1}_{\set{x (v_i^n) = v_j^n }}\right)_{i,j=1}^{2^n}.$$ In other words, $(i, j)$th entry of $A_x$ is equal to~1, if a transformation $x$\ maps  $v_i^n$ to $v_j^n$, and~0, otherwise.
\begin{remark}
In an automorphism group of a tree such a matrix describes completely the action of an automorphism. Unfortunately, for a semigroup this is not the case.
\end{remark} 
\begin{example}\label{example:matrix} Consider the partial automorphism  $x \in \mathcal{P}_2$, which  acts in the following way

	\begin{xy}
	0;<2.5cm,0cm>:<0cm,2.5cm>::
	(2.2,2)*{v_1^0}="v0";
	(1.2,1)*{v_1^1}="v11" **@{-}; ?(.23)*{}="l";
	"v0"; (3.2,1)*{v_2^1}="v12" **@{-}; ?(.23)*{}="v";
	"v11"; (0.5,0)*{v_1^2} **@{-}; 
	"v11"; (1.9,0)*{v_2^2} **@{-}; 
	"v12"; (2.5,0)*{v_3^2} **@{.}; 
	"v12"; (3.9,0)*{v_4^2} **@{.}; 
	"l"; "v" **@{-} ?>* @{>};  ?<* @{<}
	\end{xy}
	
(dotted lines mean that these edges are not in domain of $x$).
	
Then the corresponding matrix for $x$ is
	$$A_{x}=\begin{pmatrix}
	0&0&1&0\\
	0&0&0&1\\
    0&0&0&0\\
    0&0&0&0
	\end{pmatrix}.$$
Note that if $v_2^1$ were not in the $\dom(x)$ with action on other vertices preserved, then the corresponding matrix would be the same.
\end{example}

Let $\chi_x(\lambda)$ be the characteristic polynomial of $A_x$ and $\lambda_1, \ldots, \lambda_{2^n}$ be its roots respecting multiplicity. Denote $$\Xi_n = \frac{1}{2^n} \sum\limits_{k=1}^{2^n} \delta_{\lambda_k}$$ the uniform distribution on eigenvalues of $A_x$.

\begin{theorem}\label{theorem:main}
For any function $f \in C(D)$, where $D= \{z \in \mathbb{C}\mid |z|\leq 1\}$ is a unit disc, \begin{equation}\label{eq:main}
\int_D f(x)\,\Xi_n (dx) \overset{\mathbb P}{\longrightarrow} f(0),\ \ \ n\rightarrow \infty.\end{equation}
In other words, $\Xi_n$ converges weakly in probability to $\delta_0$ as $n \to \infty$, where $\delta_0$ is the delta-measure concentrated at $0$.\end{theorem}

\begin{remark}
Evans \cite{Evans} has studied asymptotic behaviour of a spectral measure of a randomly chosen element $\sigma$ of $n$-fold wreath product of symmetric group $\mathcal{S}_d$. 

He considered the random measure $\Theta_n$ on the unit circle $C$, assigning equal probabilities to the eigenvalues of $\sigma$.

Evans has shown that if $f$ is a trigonometric polynomial, then $$\lim_{n\to \infty}\mathbb{P}\left\{\int_C f(x)\, \Theta_n(dx)\neq \int f(x)\, \lambda(dx)\right\}=0,$$ 
where $\lambda$ is the normalized Lebesgue measure on the unit circle. 
Consequently, $\Theta_n$   converges weakly in probability
to $\lambda$ as $n\to\infty$.
\end{remark}

In fact, Theorem~\ref{theorem:main} speaks about the number of non-zero roots of characteristic polynomial~$\chi_x(\lambda)$. Let us  find an alternative description for them.
Denote 
$$
S_n(x) =\bigcap_{m\geq 1} \dom(x^m)=\set{v^n_j \mid v^n_j\in \dom(x^m)\text{ for all }m\ge 1}
$$
the vertices of the $n$th level, which ``survive'' under the action of $x$, and define the \textit{ultimate rank} of $x$ by $\rk_n(x) = \abs{S_n(x)}$. 
Let $R_n$ denote the total number of these vertices over all $x \in \mathcal{P}_n$,  that is $$R_n = \sum\limits_{x \in \pn{n}}^{} \rk_n(x).$$ 
We call the number  $R_n$ the \emph{total ultimate rank}.

\begin{lemma}\label{lemma:ult-rank}
For $x\in \pn{n}$ the number of non-zero roots of $\chi_x$ with regard for multiplicity is equal to the ultimate rank $\rk_n(x)$ of $x$.
\end{lemma}
\begin{proof}
Let $x\in \pn{n}$ and $A_x$ be its action matrix. Consider $A_x$ as a matrix in a standard basis. Let $w$ be some basis vector. It follows from the definition of $A_x$ that there are two possibilities: if the vertex $v$ corresponding to $w$ is in domain of $x$, then  $A_x$ sends $w$ to another basis vector, otherwise, to the zero vector.  Since $x$ is a partial bijection, applying $A_x$ repeatedly, we can either get the same vector or the zero vector; $A_x^n w =0$ means that $v\notin\dom x^n$. In the first case, the vector $w$ corresponds to a non-zero root of $\chi_x$ (some root of unity), and the vertex $v$ contributes to the ultimate rank. In the second case, the vector is a root vector for the zero eigenvalue, so it corresponds to a zero root of $A_x$, while the corresponding vector does not contribute to the ultimate rank. 
\end{proof}

 Denote $\rank_n(x)=|\dom(x)\cap B_n|$  and define the \textit{total rank} $$R_n'=\sum\limits_{x \in \pn{n}} \rank_n(x).$$
 \begin{remark}\label{remark:trank}
Clearly, for $x=(f,a)$, where $a\in \is{2}$, $f\colon \dom(a)\ra \pn{n-1}$, \begin{equation}\label{eq:recurrank}
\rank_n(x)=\sum_{y\in \dom(a)}\rank_{n-1}(f(y))
\end{equation}
if $\dom(a)\neq\varnothing$ and $\rank_n(x)=0$ otherwise.
 \end{remark} 

\begin{lemma}\label{lemma:sum-rank} 
Let $R'_n$ be the  total rank of the semigroup $\mathcal{P}_n$. Then $$R_n' = 4R_{n-1}'+ 4R_{n-1}' N_{n-1}.$$
\end{lemma}

\begin{proof} Thanks to \eqref{eq:recurrank}
\begin{gather*}
R_n' = \sum_{\substack{x =(f,a)\in \mathcal{P}_n}} \rank_n(x)=\sum_{\substack{x =(f,a)\in \mathcal{P}_n}}\sum_{\substack{y \in \dom(a)}} \rank_{n-1}(f(y))\\= \sum_{\substack{a\in \is{2}\\\abs{\dom(a)=1}}}\sum_{\substack{f_1\in \pn{n-1}}}\rank_{n-1}(f_1)\\+\sum_{\substack{a\in \is{2}\\\abs{\dom(a)=2}}}\sum_{\substack{f_1, f_2\in \pn{n-1}}}\left(\rank_{n-1}(f_1)+\rank_{n-1}(f_2)\right)\\=4R_{n-1}'+2\sum_{\substack{f_1, f_2 \in \pn{n-1}}}\left(\rank_{n-1}(f_1)+\rank_{n-1}(f_2)\right).
\end{gather*}

By symmetry, \begin{gather*}
\sum_{\substack{f_1, f_2 \in \pn{n}}}\rank_{n-1}(f_1)=\sum_{\substack{f_1\in \pn{n-1}}}\sum_{\substack{f_2\in \pn{n-1}}}\rank(f_1)\\=\sum_{\substack{f_1, f_2 \in \pn{n}}\in \pn{n-1}}N_{n-1}\rank(f_1)=N_{n-1}R'_{n-1}.
\end{gather*}
Hence, $R_n'=4R_{n-1}'+4R_{n-1}'N_{n-1}$.
\end{proof}

\begin{lemma}\label{lemma:sum-rank-again} Let $R'_n$ be the total rank of the semigroup $\mathcal{P}_n$. Then $$R_n' = 2^{n-1}(1+N_n)=2^{2^{n}+n-2}.$$\end{lemma}

\begin{proof} We proceed by induction. A direct calculation gives $$R_1'= 8=1+N_1.$$ 

Assuming that 
$$
R_{n-1}' = 2^{2^{n-1}+n-3}, 
$$
we have, thanks to Lemma~\ref{lemma:sum-rank} and Proposition~\ref{proposition:card},  
$$ 
R_n' = 4R'_{n-1}(1+N_{n-1}) = 4\cdot 2^{2^{n-1}+n-3}\cdot 2^{2^{n-1}-1} = 2^{2^{n}+n-2},
$$
as required.
\end{proof}

\begin{lemma}\label{lemma:ev-rank}
Let $R_n$ be the total ultimate rank of the semigroup $\mathcal{P}_n$. Then
$$R_n \leq 3 R_{n-1}+3R_{n-1}N_{n-1}.$$
\end{lemma} 
\begin{proof} 
Represent $R_n$ as a sum
$$R_n=\sum_{\substack{x=(f,a)\in\pn{n}}}\rk_n(x)=\sum_{\substack{x=(f,a)\in\pn{n}\\\abs{\dom(a)}=1}}\rk_n(x)+\sum_{\substack{x=(f,a)\in\pn{n}\\\abs{\dom(a)}=2}}\rk_n(x).$$

If  $\rank(a)= 1$, then we will be interested only in those $a$ for which $a=(1)$ and $a = (2)$, since otherwise the ultimate rank of $x$ is 0. Therefore,  
\begin{gather*}
\sum\limits_{\substack{x=(f,a) \in \pn{n}\\|\dom(a)|=1}} \rk_n(x)=\sum_{a\in \set{(1),(2)}}\sum_{\substack{f_1}\in\pn{n-1}}\rk_{n-1}(f_1)\\=2 \sum\limits_{f_1 \in \pn{n-1}}\rk_{n-1}(f_1)= 2R_{n-1}.
\end{gather*}

If  $\rank(a)= 2$, then
\begin{gather*}\sum_{\substack{{x=(f,a) \in \pn{n}}\\|\dom(a)|=2}}^{}  \rk_n(x) = \sum_{\substack{{x=(f,a) \in \pn{n}}\\a=(1)(2)}} \rk_n(x) + \sum_{\substack{{x=(f,a) \in \pn{n}}\\a=(12)}} \rk_n(x) =: S_1 + S_2.\end{gather*}

Clearly, if $x=(f,a)$ with $a=(1)(2)$, then $\rk_n(x)=\rk_{n-1}(f(1))+\rk_{n-1}(f(2))$, whence
\begin{gather*}
S_1 = \sum_{\substack{f_1, f_2\in \pn{n-1}}}\left(\rk_{n-1}(f_1)+\rk_{n-1}(f_2)\right)\\=2\sum_{\substack{f_1, f_2}\in \pn{n-1}}\rk_{n-1}(f_1)=2 R_{n-1}N_{n-1}.
\end{gather*}

Further, if $x=(f,a)$ with $a=(12)$, then $\rk_{n}(x)=2\rk_{n-1}(f(1)(f(2))$. So, 
\begin{gather*}
S_2=2\sum_{\substack{f_1, f_2\in \pn{n-1}}}\rk_{n-1}(f_1 f_2).
\end{gather*}

Note that every element  $x \in \pn{n}$ can be represented as a product $x = e \sigma$, where $e$ is an idempotent on $\dom(x)$ and $\sigma$ is a permutation on the set of the tree vertices.
Then $$\sum\limits_{x_2 \in \pn{n-1}}\rk_{n-1}(x_1 x_2) =\sum\limits_{x_2 \in \pn{n-1}}\rk_{n-1}(e \sigma x_2)= \sum\limits_{x_2 \in \pn{n-1}}\rk_{n-1}(e x_2).$$
The last equality is true since transformation $x\mapsto \sigma x$ is bijective on $\pn{n-1}$.

It follows from above that
\begin{gather*}
S_2=2\sum_{\substack{f_1, f_2\in \pn{n-1}}}\rk_{n-1}(f_1 f_2) = 2  \sum\limits_{f_1 \in \pn{n-1} }\sum\limits_{f_2 \in \pn{n-1}} \rk_{n-1}(\id_{\dom(f_1)} f_2)\\
\leq 2  \sum\limits_{f_1 \in \pn{n-1} }\sum\limits_{f_2 \in \pn{n-1}}\abs{\dom(f_1) \cap S_{n-1}(f_2)} \\
= 2  \sum\limits_{f_1 \in \pn{n-1} }^{}    \sum\limits_{f_2 \in \pn{n-1}}\sum\limits_{j=1}^{2^{n-1}} \mathbf{1}_{\{v_j^{n-1} \in \dom(f_1)\}}\cdot\mathbf{1}_{\{v_j^{n-1} \in S_{n-1}(f_2)\}}\\
=  2  \sum\limits_{j=1}^{2^{n-1}} \sum\limits_{f_1 \in \pn{n-1} }^{} \mathbf{1}_{\{v_j^{n-1}\in \dom(f_1)\}} \cdot  \sum\limits_{f_2 \in \pn{n-1}} \mathbf{1}_{\{v_j^{n-1}\in S_{n-1}(f_2)\}}.
\end{gather*} 
Thanks to symmetry, for each $j$
$$\sum\limits_{x \in \pn{n}}\mathbf{1}_{\{v_j^{n-1} \in \dom(x)\}}=\sum\limits_{x \in \pn{n}} \mathbf{1}_{\{v_1^{n-1} \in \dom(x)\}}.$$
Therefore, $$\frac{1}{2^n} \sum\limits_{x \in \pn{n}} \sum\limits_{k=1}^{2^{n-1}} \mathbf{1}_{\{v_k^{n-1} \in \dom(x)\}} = \frac{1}{2^{n-1}} \abs{\dom(x)}.$$ 
Thus, we can write 
\begin{gather*}
S_2=2 \sum\limits_{j=1}^{2^{n-1}} \sum\limits_{f_1 \in \pn{n-1}} \frac{1}{2^{n-1}} \abs{\dom(f_1)} \mathbf{1}_{\{v_j^{n-1} \in S_{n-1}(f_2)\}}\\ = \frac{1}{2^{n-1}} \sum\limits_{f_1 \in \pn{n}}\abs{\dom(f_1)} \sum\limits_{j=1}^{2^{n}} \mathbf{1}_{\{v_j^{n-1} \in S_{n-1}(f_2)\}}\\
= 2 \cdot  \frac{1}{2^{n-1}} \sum\limits_{f_1 \in \pn{n-1}} \abs{\dom(f_1)}\cdot \sum\limits_{f_2 \in \pn{n-1}} \abs{S_{n-1}(f_2)}.\end{gather*}

Using that $ \abs{S_{n-1}(f_2)}=\rk_{n-1}(f_2)$, $ \abs{\dom(f_1)} = \rank(f_1)$,  and applying Lemma~\ref{lemma:sum-rank-again}, we  get
\begin{gather*}
\frac{2}{2^{n-1}}\, R_{n-1} R_{n-1}' = \frac{2 R_{n-1}\cdot (1 + N_{n-1}) 2^{n-2}}{2^{n-1}}\\= \frac{2R_{n-1}(1 + N_{n-1}) }{2}= (1 + N_{n-1}) R_{n-1}.
\end{gather*}

Therefore, $R_n \leq 2 R_{n-1} + 2 R_{n-1} N_{n-1} + (1 + N_{n-1}) R_{n-1}= 3 R_{n-1}+ 3R_{n-1}N_{n-1}.$ \end{proof}


\begin{lemma}\label{lemma:estpn}
For $n \in \mathbb{N}$ denote $p_n=\dfrac{R_n}{2^n N_n}$. Then $$p_n\leq \dfrac{\ 3\ }{\ 4\ } p_{n-1}, \quad n \geq 2.$$
\end{lemma}
\begin{proof}
Using Lemma~\ref{lemma:ev-rank}, we get
\begin{gather*}
p_n = \frac{R_n}{2^n N_n} \leq 
\frac{3 R_{n-1} + 3 R_{n-1} N_{n-1}}{2^n N_n} = \frac{3 R_{n-1}(1+ N_{n-1})}{2^n N_n}\\
= \frac{3 R_{n-1}\cdot 2^{2^{n}-1}}{2^n (2^{2^{n+1}-1}-1)}=  \frac{3}{2} \cdot \frac{R_{n-1}}{2^{n-1} (2^{2^{n}}-2^{1-2^n})} 
 \leq \frac{3}{2} \cdot \frac{R_{n-1}}{2^{n-1} (2^{2^{n}}-2)}\\ = \frac{3}{2} \cdot \frac{R_{n-1}}{2^{n} (2^{2^{n}-1}-1)} 
 = \frac{3}{4} \cdot \frac{R_{n-1}}{ 2^{n-1} N_{n-1}}=   \frac{3}{4} \cdot p_{n-1}.\qedhere\end{gather*} \end{proof}
\begin{proof}[Proof of Theorem \ref{theorem:main}]
Note that $\int_D f(z)\Xi_n (dz)=\frac{1}{2^n}\sum_{k=1}^{2^n}f(\lambda_k)$, where $\lambda_1, \ldots, \lambda_{2^n}$ are the roots of characteristic polynomial $\chi_x(\lambda)$.
	
Then, thanks to Lemma~\ref{lemma:ult-rank} \begin{gather*}
\abs{\int_D f(z)\Xi_n (dz) - f(0)}=\abs{\int_D\left(f(z)-f(0)\right)\Xi(dz)}\\\leq\frac{1}{2^n}\sum_{k:\lambda_k\neq 0} \abs{\left(f(k)-f(0)\right)}\leq
2\max_{D}\abs{f}\cdot\frac{\abs{k:\lambda_k\neq 0}}{2^n}=2  \max_{D}\abs{f} \frac{\rk_n(x)}{2^n}
\end{gather*}

Therefore, 
\begin{gather*}
\mathbb{E}\abs{\int_D f(z)\Xi(dz)-f(0)}\leq 2\max\abs{f}\cdot \frac{\sum_{z\in \pn{n}}\rk_n(x)}{2^n N_n}=2\max\abs{f}\cdot p_n,
\end{gather*}
where $p_n$ is defined in Lemma~\ref{lemma:estpn}.
	By Lemma~\ref{lemma:estpn}, $p_n \leq \frac{3}{4} \cdot p_{n-1}\leq \ldots \leq (\frac{3}{4})^{n-1}p_0$, whence $p_n\to 0$, $n\to \infty$.

Consequently, 
$\mathbb{E}\abs{\int_D f(z)\Xi(dz)-f(0)}\to 0$, $n \to \infty$, whence the statement follows.
\end{proof}
\begin{remark}
We can see from the proof that the rate of convergence in \eqref{eq:main} is in some sense exponential.
\end{remark}

\end{document}